\documentclass[fullpage]{article}

\usepackage{amsmath,amsthm,amsfonts,amssymb,amscd,epsf,epsfig,psfrag,enumerate, tikz,bbold}
\usepackage[linesnumbered]{algorithm2e}

\def\intt{\operatorname{int}}
\def\cl{\operatorname{cl}}
\def\cone{\operatorname{cone}}

\def\vol{\operatorname{vol}}

\def\epi{\operatorname{epi}}
\DeclareMathOperator*{\argmin}{argmin}
\DeclareMathOperator*{\argmax}{argmax}

\def\w{\omega}
\def\R{\mathbb{R}}\def\Z{\mathbb{Z}}\def\N{\mathbb{N}}

\def\T{\mathsf{T}}
\def\C{\mathcal{C}}
\def\F{\mathcal{F}}

\def\d{\,\mathrm{d}}

\newtheorem{theorem}{Theorem}
\newtheorem{lemma}[theorem]{Lemma}

\newtheorem{definition}[theorem]{Definition}
\newtheorem{remark}[theorem]{Remark}
\newtheorem{corollary}[theorem]{Corollary}
\newtheorem{conjecture}[theorem]{Conjecture}

\newtheorem{claim}[theorem]{Claim}

\title{Centerpoints: A link between optimization and convex geometry}
\author{Amitabh Basu \and Timm Oertel}
 
\begin{document}

\maketitle

\begin{abstract}
  We introduce a concept that generalizes several different notions of a ``centerpoint'' in the literature. We develop an oracle-based algorithm for convex mixed-integer optimization based on centerpoints. Further, we show that algorithms based on centerpoints are ``best possible'' in a certain sense. Motivated by this, we establish several structural results about this concept and provide efficient algorithms for computing these points. Our main motivation is to understand the  complexity of oracle based convex mixed-integer optimization.
\end{abstract}

\section{Introduction}\label{sec:intro}
Consider the following unconstrained optimization problem
\begin{equation}\label{pro:ConvexMixedIntegerOpt}
\min\limits_{(x,y)\in\Z^{n}\times\R^d} g(x,y),
\end{equation}
where $g:\R^n\times\R^d\to\R$ is a convex function. To keep the problem as general as possible, we assume that $g$ can be accessed only by a first-order evaluation oracle. In other words, when queried at a point $(x,y)$, the oracle returns the corresponding function value $g(x,y)$ and an element from the subdifferential $\partial g(x,y)$. This allows us to model very general, possibly non-smooth, convex functions. The only additional assumption we make to keep \eqref{pro:ConvexMixedIntegerOpt} tractable, is that the minimization problem is bounded.

We present a general cutting plane algorithm based on the concept of {\it centerpoints}, which we define below.
We call it the {\it centerpoint algorithm}.
Our approach bears similarities to a number of continuous convex minimization algorithms and to Lenstra-type algorithms \cite{Lenstra83,GroetschelLovaszSchrijver-Book88} for convex integer optimization problems. Most variations of Lenstra-type algorithms rely on a combination of the ellipsoid method and enumeration on lower dimensional subproblems \cite{Kannan87,Heinz05,MicciancioVoulgaris2010,hildebrand2013new,OertelWagnerWeismantel14,Dadush12}. The key difference is that our algorithm avoids enumerating low dimensional subproblems.

The main feature of this approach is that, from the point of view of the number of function oracle calls, this algorithm is best possible, up to a lower order factor. For this purpose, we present our results for a somewhat general convex optimization problem (see Section~\ref{sec:center-opti} for details), and then specialize the results to continuous/integer/mixed-integer convex optimization. We now proceed to the central concept of this paper.

\subsection*{Centerpoints} Let $\mu$ be a Borel-measure\footnote{{  A reader  unfamiliar with measure theory may simply consider $\mu$ to be the volume measure or the mixed-integer measure on the mixed-integer lattice, i.e., $\mu(C)$ returns the volume of $C$ or the ``mixed-integer volume'' of the mixed-integer lattice points inside $C$. 
The ``mixed-integer volume'' reduces to the number of integer points when the lattice is the set of integer points. See~\eqref{eq:mixed} for a precise definition which generalizes both the standard volume and the ``counting measure'' for the integer lattice.}} on $\R^n$ such that $0<\mu(\R^n)<\infty$. Without any loss of generality, we normalize the measure to be a probability measure, i.e., $\mu(\R^n)=1$. For $S\subseteq \R^n$ nonempty and closed, we define the set of {\em centerpoints} $\C(S,\mu)\subseteq S$ as the set that attains the following maximum
\begin{equation}\label{eq:max-value}\F(S,\mu):=\max_{x\in S}\inf_{u\in \mathcal{S}^{n-1}}\mu(H^+(u,x)),\end{equation}
where $\mathcal{S}^{n-1}$ denotes the $(n-1)$-dimensional unit sphere with center in the origin and $H^+(u,x)$ denotes the half-space $\left\{y\in\R^n : u \cdot (y-x)\ge 0 \right\}$. In other words, $\F(S,\mu)$ is the largest real number $M > 0$, such that there is a point $x \in S$ 
with the property that any halfspace containing $x$ has measure at least $M$. 
This definition unifies several definitions from convex geometry, computer science and statistics. Two notable examples are:

\begin{enumerate}
\item {\em Winternitz measure of symmetry.} Let $\mu$ be the Lebesgue measure restricted to a  convex body $K$ {   (i.e., $K$ is compact and has a non-empty interior)}, or equivalently, the uniform probability measure on $K$, and let $S=\R^n$. $\F(S,\mu)$ in this setting is known in the literature as the {\em Winternitz measure of symmetry} of $K$, and the centerpoints $\mathcal{C}(S,\mu)$ are the ``points of symmetry'' of $K$. This notion was studied by Gr\"unbaum in~\cite{Gruenbaum1960} and surveyed by the same author in~\cite{grunbaum1963measures}. Caplin and Nalebuff~\cite{caplin198864} generalize Gr\"unbaum's results to measures $\mu$ with a concave density supported on a compact set $K$. \cite{toth2015measures} is a recent survey on measures of symmetry of convex bodies. Convex geometry literature also studies the closely related concepts of floating bodies and illumination bodies; see~\cite{werner2006floating} for a survey.

\item {\em Tukey depth and median.} In statistics and computational geometry, the function $f_\mu:\R^n \to \R$ defined as
\begin{equation}\label{eq:f}f_\mu(x) :=\inf_{u \in \mathcal{S}^{n-1}} \mu(H^+(u,x)) \end{equation} 
is known as the {\em halfspace depth function} or the {\em Tukey depth function} for the measure $\mu$, first introduced by Tukey~\cite{tukey1975mathematics}. Taking $S =\R^n$, the centerpoints $\mathcal{C}(\R^n,\mu)$ are known as the {\em Tukey medians} of the probability measure $\mu$, and $\F(\R^n,\mu)$ is known as the maximum {\em Tukey depth} of $\mu$. Tukey introduced the concept when $\mu$ is a finite sum of Dirac measures (i.e., a finite set of points with the counting measure), but the concept has been generalized to other probability measures and analyzed from structural,~\cite{donoho1992breakdown,rousseeuw1999depth,koshevoy2002tukey}, 
as well as computational perspectives~\cite{rousseeuw1999depth,chan2004optimal,bremner2008output,dyckerhoff-mozharovski}. 
See~\cite{mosler2013depth} for a survey of structural aspects and other notions of ``depth'' used in statistics, and~\cite{dyckerhoff-mozharovski} and the references therein for a survey and recent approaches to computational aspects of the Tukey depth. 
%
%
%
\end{enumerate}

\subsection*{Our Results\footnote{Earlier versions of this paper from {\tt arxiv.org} and {\em IPCO 2016} contained results about the uniqueness of the centerpoint. We were recently made aware by an anonymous referee that such uniqueness results already existed in the convex geometry literature~\cite[Proposition 1]{werner2006floating}. While our uniqueness result had been obtained independently and without knowledge of these prior results, the proof ideas used by us were quite similar to the existing proof. Consequently, we do not find any value in including these results in this version of the paper.}}
To the best of our knowledge, all related notions of centerpoints in the literature have only considered the case where 
the set $S$ is $\R^n$, i.e., the centerpoint can be any point from the Euclidean space. We consider more general $S$, as this allows us to analyze convex optimization problems where the solutions have to satisfy side constraints like mixed-integer constraints, sparsity constraints (e.g. compressed sensing), or complementarity constraints. Essentially, the closed subset $S$ is going to represent non-convex feasibility constraints of our optimization problem; at the most general level, all we require from $S$ is that it is closed and nonempty.

In Section~\ref{sec:center-opti}, we elaborate on this connection between centerpoints and algorithms for optimizing convex functions over general closed sets $S$. We first give an algorithm for solving such problems given access to first-order (separation) oracles, based on centerpoints. We then focus on convex {\em mixed-integer} optimization and show that the centerpoint algorithms is ``best possible'' in a certain sense, amongst a large class of first-order oracle based methods -- see Table~\ref{fig:M1}. This comprises our main motivation in studying centerpoints. 

In Section~\ref{sec:general-properties}, we provide lower bounds on the value of $\F(S,\mu)$. In Section~\ref{sec:general:thm:helly}, we obtain lower bounds in terms of the Helly number of $S$ with minimal assumptions on $S$ and $\mu$. In Section~\ref{sec:new-structure}, we obtain better lower bounds for the special case when $S = \Z^n\times\R^d$ and $\mu$ is the ``mixed-integer" uniform measure on $K \cap (\Z^n\times\R^d)$, where $K$ is some convex body. Such bounds immediately imply bounds on the complexity of oracle-based convex mixed-integer optimization algorithms.

%

In Section~\ref{sec:algorithms}, we present a number of exact and approximation algorithms for computing centerpoints. To the best of our knowledge, the computational study of centerpoints has only been done for measures $\mu$ that are a finite sum of Dirac measures, i.e., for finite point sets, or when $\mu$ is the uniform measure on two dimensional polygons (e.g. see \cite{brass2003computing} and the references therein). We initiate a study for other measures; in particular, the uniform measure on a  convex body, the counting measure on the integer points in a  convex body, and the mixed-integer volume of the mixed-integer points in a  convex body. All our algorithms are exponential in the dimension $n$ but polynomial in the remaining input data, so these are polynomial time algorithms if $n$ is assumed to be a constant. Algorithms that are polynomial in $n$ are likely to not exist because of the reduction to the so-called {\em closed hemisphere problem} -- see Chapter~7 by Bremner, Fukuda and Rosta in the collection of articles in~\cite{liu2006data}.

We mention that the algorithms for computing centerpoints from Section~\ref{sec:algorithms} are based on the standard Turing machine model of computation and, therefore, work with rational arithmetic. Consequently, since the coordinates of a centerpoint could be irrational, our algorithms return points whose coordinates are arbitrary close approximations of the centerpoint coordinates. We alert the reader that our analysis of the oracle complexity of cutting plane algorithms in Section~\ref{sec:center-opti} ignores such arithmetic issues and the results assume that exact centerpoints are used in the {\em optimization algorithm}. A general framework for handling such arithmetic issues is described in~\cite{GroetschelLovaszSchrijver-Book88}. For this reason, we do not discuss them further in this manuscript.

Throughout this paper we use the notation $H^-(u,x):=\{y\in\R^n : u\cdot(y-x)\le0\}$ and $H(u,x):=\{y\in\R^n : u\cdot(y-x)=0\}$ for $u\in \mathcal{S}^{n-1}$ and $x\in\R^n$. Recall that $H^+(u,x)$ denotes the half-space $\left\{y\in\R^n : u \cdot (y-x)\ge 0 \right\}$. With $\intt(X)$ we will denote the interior of $X \subseteq \R^n$.



\section{The connection to optimization}\label{sec:center-opti} 
Given a nonempty, closed set $S \subseteq \R^n$, consider the following optimization problem:
\begin{equation}\label{pro:general-Convex-Opt}
\min\limits_{x \in S} g(x),
\end{equation}
where $g:\R^n\to\R$ is a convex function given by a first-order evaluation oracle. We first define the class of algorithms against which we will compare the centerpoint-algorithm. We refer to this class as {\em cutting plane algorithms}.



\begin{definition}[CUTTING PLANE ALGORITHM]\label{def:cutting-plane}
Let $S\subseteq \R^n$ be a nonempty, closed set and let $\nu$ be a Borel measure on $\R^n$ such that $\nu(\R^n\setminus S) =0$, i.e., $\nu$ is supported on $S$. A {\em cutting plane algorithm} for~\eqref{pro:general-Convex-Opt} with {\em stopping criterion based on $\nu$} is an algorithm with the following structure:

\begin{quote}
{\bf INPUT:} An error guarantee $\delta>0$ and a convex 
set $E_0$ that contains the optimal solution in its interior and with $\nu(E_0) > 0$.

{\bf ITERATIONS:} At each iteration $i= 1, 2 \ldots$ the algorithm selects a point $x_i \in \intt (E_{i-1})\cap S$, and then makes a call to a first-order oracle for $g$ at $x_i$, which returns the function value $g(x_i)$ and a subgradient $h_i \in \partial g(x_i)$. We define $x^\star:=\argmin_{x \in \{x_1, \ldots, x_i\}}g(x)$ and define $E_i$ such that
\begin{equation}\label{Alg:choice:Supset}E_{i}\;\supseteq\;\{ x \in E_0 :   g(x^\star)-g(x_j)\ge h_j \cdot (x-x_j),\; \forall j=1,\dots,i \}.
\end{equation}

{\bf OUTPUT:} The cutting plane algorithm stops at the $N$-th iteration when $\nu(\intt(E_{N})) \leq \delta$ and returns $x^\star$.
\end{quote}
\end{definition}


When $S = \R^n$ and $\nu$ is the standard Lebesgue measure, we obtain the standard cutting plane algorithms for continuous convex optimization, such as the Ellipsoid method, Method of centers of gravity, Kelly method or the Level method \cite[Section 3]{Nesterov-Book04}.
A variant of the cutting plane method which utilizes random sampling was explored by Bertsimas and Vempala in~\cite{BertsimasVempala04} for continuous convex optimization. Their method also falls under the general framework of Definition~\ref{def:cutting-plane}.
When $S = \Z^n$ and $\nu$ is the counting measure for $\Z^n$, we obtain cutting plane algorithms for convex integer optimization problems. When $S = \Z^n \times \R^d$, we obtain algorithms for convex mixed-integer optimization. A natural choice of the measure is the mixed-integer measure $\bar\mu_{mixed}$, which we will discuss in the next section.

\begin{remark} We give a short justification of our stopping criterion in Definition~\ref{def:cutting-plane}. Given general convex functions and the absence of any known structure for $S$, one can only guarantee an approximate solution from an algorithmic point of view in general (for structured $S$ like $\Z^n$, the situation is different). Typically, the quality of such approximations is then quantified by an additive or multiplicative gap with respect to the optimal function value. These often require additional estimations based on further parameters, for example, a Lipschitz constant.
Instead of considering the gap in the function values, we approximate an optimal point and we quantify the approximation quality by an appropriate measure $\nu$, thus circumventing any additional assumptions or sources of error such as Lipschitz constants. 
At the end of this section, we elaborate on how to extend our results to derive bounds on the additive gap with respect to the optimal function value in the mixed-integer case.
\end{remark}
 
\begin{remark} We also briefly comment on the assumption that $\nu(\R^n\setminus S) = 0$, i.e., the measure is supported on $S$. One could consider more general measures that do not satisfy this condition and analyze the class of algorithms obtained thus. However, the mathematical analysis of this more general situation becomes more tedious with many corner cases that need to be handled, without adding any new insight. It is more elegant to restrict the analysis to measures that are supported on $S$. Consequently, we build this into the definition of our cutting plane algorithm.
\end{remark}

\begin{definition}[CENTERPOINT ALGORITHM]
Let $S\subseteq \R^n$ be a nonempty, closed set and let $\nu$ be a Borel measure on $\R^n$ such that $\nu(\R^n\setminus S)=0$. The {\em centerpoint algorithm} is a cutting plane algorithm for~\eqref{pro:general-Convex-Opt} that chooses $x_i \in \mathcal{C}(S,\nu_i)$ from the set of centerpoints, where $\nu_i$ is the measure $\nu$ restricted to $\intt(E_{i-1})$, and defines $E_i$ to be the right hand side of~\eqref{Alg:choice:Supset}.
\end{definition}

For our general bounds we need three parameters related to $S$ and $\nu$.

\begin{definition}\label{def:regularity} Let $S\subseteq \R^n$ be a nonempty, closed subset and $\nu$ be a measure on $\R^n$ that is finite for any bounded set. 
\begin{enumerate}[(i)]
\item 
For any bounded convex set $C \subseteq \R^n$ with $\nu(C)>0$, define $\nu_C$ as the normalized finite measure $\nu$ restricted to $C$. We define  $$c(S, \nu):= \inf_{\nu_C} \mathcal{F}(S, \nu_C).$$
\item We define the {\it degeneracy parameter} $$\chi(S,\nu):=\max_{y \in S} \min_{u\in\mathcal{S}^{n-1}}\nu(\{x\in \R^n : u\cdot x = u\cdot y\}).$$
\item We say that $\nu$ is lower semi-continuous if for every $\epsilon>0$ and every open set $A \subseteq \R^n$, there exists a closed set $A' \subseteq A$ such that $$\nu(\intt (A'))\ge\nu(A) - \epsilon.$$
\end{enumerate}
\end{definition}

In Theorem~\ref{thm:general:thm:helly}, we will show that $c(S, \nu) \ge h(S)^{-1}$, where $h(S)$ denotes the Helly number of $S$. However, for certain types of measures one can obtain stronger bounds, e.g., see Corollary~\ref{cor:continuous-bound}.
Note that the parameter $\chi$ is zero when $S$ corresponds to the whole space, a mixed-integer lattice or to sets obtained from sparsity constraints on the variables. However, if $S=\Z^n$ and $\nu$ corresponds to the counting measure, then $\chi(S,\nu)=1$.

Our first general result showing the asymptotic optimality of the centerpoint based algorithm amongst cutting plane algorithms for~\eqref{pro:general-Convex-Opt} is the following. 

\begin{theorem}\label{thm:general-optimality}[General optimality bounds]
Let $S$ be a nonempty, closed set. Let $\nu$ be a measure such that $\nu(\R^n\setminus S) = 0$, $\nu$ is finite for bounded sets, and $\nu$ is lower semi-continuous (as defined in part (iii) of Definition~\ref{def:regularity}). Further, assume that $c(S,\nu) > 0$. Let $\delta > 0$ and $E_0\subseteq \R^n$ with $ \nu(E_0) = V > 0$. The centerpoint algorithm makes at most $$\bigg\lceil{\log_{\frac{1}{1-c(S,\nu)}}\left(\frac{V}{\delta}\right)}\bigg\rceil$$ first-order oracle calls for any convex function $g$. Moreover, for any cutting plane based algorithm $\mathcal{A}$ with a  stopping criterion based on $\nu$, there exists a convex function $\hat g$ such that $\mathcal{A}$ will make at least $$\bigg\lceil\left(\log_{2}\left(\frac{V}{\delta + \chi(S,\nu)}\right)\right)\bigg\rceil - 1$$ first-order evaluation oracle calls to $\hat g$.
\end{theorem}

\begin{proof}
The upper bound follows from the fact that by choosing the centerpoint at every iteration, one can guarantee that $\nu(\intt(E_i)) \leq (1 - c(S,\nu))\nu(\intt(E_{i-1}))$ for every $i=1, \ldots, N$, where $N$ is the number of iterations such that $\nu(\intt(E_{N})) \leq \delta$.

For the lower bound, it suffices to establish the following claim whose proof appears below. \begin{claim}\label{claim:lower-bound}For any cutting plane algorithm $\mathcal{A}$ and a real number $\epsilon > 0$, there exists a sequence of convex functions $\{g_i\}_{i=1}^\infty$ such that for every $i \geq 1$, if $\mathcal{A}$ runs for $i$ iterations on $g_i$, then $\nu(\intt(E_{i})) \geq (\frac12)^{i} \nu(\intt(E_0)) - 2(\epsilon - \chi(S,\mu))$.\end{claim} To obtain the lower bound stated in the theorem from this, we do the following: Given $\delta>0, V>0$, set $N^* := \bigg\lceil\left(\log_{2}\left(\frac{V}{\delta + \chi(S,\nu)}\right)\right)\bigg\rceil - 1$ and we run $\mathcal{A}$ on $g_{N^*}$. Using Claim~\ref{claim:lower-bound} with $\epsilon = \frac{\delta}{2}$, we obtain that $\nu(\intt(E_{N^*})) \geq (\frac12)^{N^*} \nu(\intt(E_0)) - \delta - 2\chi(S,\mu))$. Since the stopping criterion for the algorithm is $\nu(\intt(E_{N^*})) \leq \delta$, the inequality implies that $\mathcal{A}$ requires at least $N^*$ iterations to stop.
%
%
%
%
%
%
%
\begin{proof}[Proof of Claim~\ref{claim:lower-bound}]\renewcommand{\qedsymbol}{$\diamond$} We construct the sequence $\{g_i\}_{i=1}^\infty$ in an adversarial manner; we will actually construct $\epi(g_i)$, the epigraphs of $g_i$\footnote{For any function $f:\R^n \to \R$, the epigraph of $f$ is $\epi(f) := \{(x,t)\in \R^n \times \R: t \geq f(x)\}.$}, and use the fact that an epigraph defines a convex function uniquely and vice versa. 

%

In fact, we will inductively construct three sequences: convex functions $\{g_i\}_{i=0}^\infty$, vectors $\{h_i\}_{i=1}^\infty \subseteq \R^n$, and real numbers $\{\xi_i\}_{i=1}^\infty \subseteq \R$ such that the following conditions hold for every $i\geq 1$.
\begin{itemize}
\item[(i)] If algorithm $\mathcal{A}$ runs on the function $g_i$ for $i$ iterations, $$\nu(\intt(E_{i})) \geq (\frac12)^{i}\nu(\intt(E_0)) - (\sum_{j=0}^{i-1}\frac{1}{2^j})\cdot(\epsilon + \chi(S,\nu)),$$
\item[(ii)] For any $0 \leq j \leq i$, when the algorithm is executed for $j$ iterations on $g_j$ and $g_i$, it queries the same points $x_1, \ldots, x_j$.
\item[(iii)] For any $0 \leq j \leq i$, $\{h_j\} = \partial g_i(x_j)$,
\item[(iv)] Let $x_1, \ldots x_i$ be the points queried by $\mathcal{A}$ when executed on $g_i$ for $i$ iterations. Then,
$$\epi(g_i) = \bigcap_{j=1}^{i}\{(x,t): E_0 \times \R: t \geq h_j\cdot(x-x_j) - \sum_{k=1}^j\xi_k \}$$
\end{itemize}

First observe that (i) of the claim implies $\nu(\intt(E_{i})) \geq (\frac12)^{i}\nu(\intt(E_0)) - (\sum_{j=0}^{i-1}\frac{1}{2^j})\cdot(\epsilon + \chi(S,\nu))\geq (\frac12)^{i} \nu(\intt(E_0)) - 2(\epsilon + \chi(S,\mu))$. This would complete the proof of Claim~\ref{claim:lower-bound}.

We prove the claim inductively. Let $x_1 \in \intt (E_0)$ be the first point queried by $\mathcal{A}$ on any convex function. Choose $h_1\in\R^{n}\setminus \{0\}$ 
such that $\nu(\{x\in \R^n : h_1\cdot x = h_1\cdot x_1\})\le\chi(S,\nu)$, and that $\nu(\intt(E_{0})\cap \{x : h_1\cdot(x - x_1) \leq 0\})\ge\frac{1}{2}\nu(\intt(E_{0}))$. Such a choice of $h_1$ always exists. Set $\xi_1 = 0$.
 Finally, define $\epi(g_1):=\{(x,t): E_0 \times \R: t \geq h_1\cdot(x-x_1) \}.$ One can now check that (i)-(iv) in the above claim hold for $i=1$.

Now, suppose we have defined $g_1, \ldots, g_{i}$, $h_1, \ldots, h_{i}$ and $\xi_1, \ldots, \xi_{i}$ for some $i\geq 1$ such that (i)-(iv) hold. We now construct $g_{i+1}, h_{i+1}$ and $\xi_{i+1}$. Note that by (iii), the algorithm chooses $x_{i+1} \in \intt(E_{i})$ where $E_i \supseteq \widehat E_{i}:= \{ x \in E_0 :  h_j \cdot (x-x_j) \leq -\sum_{k=j+1}^i \xi_k,\; j=1,\dots,i \}$. 
If $x_{i+1} \not\in \intt(\widehat E_{i})$, then set $g_{i+1} = g_{i}$, $h_{i+1} \in \partial g_i(x_{i+1})$ (which is well-defined) and $\xi_{i+1} = 0$. Otherwise, by our assumption that $\nu$ is lower semi-continuous, one can choose $\xi_{i+1}>0$
, such that following conditions hold for $\tilde E_i:=\{ x \in E_0 :  h_j \cdot (x-x_j) \leq -\sum_{k=j+1}^{i} \xi_k - \xi_{i+1},\; j=1,\dots,i \}$:
\begin{itemize}
\item[(A)] $\nu(\intt(\widehat{E}_{i}))-\epsilon\le\nu(\intt(\tilde E_{i}))$ and
\item[(B)] $x_i\in\tilde E_{i}$.
\end{itemize}
Let $h_{i+1}\in\R^n\setminus\{0\}$ such that the following all hold:
\begin{itemize}
\item[(a)] $\nu(\{x\in \R^n : h_{i+1}\cdot x = h_{i+1}\cdot x_{i+1}\})\le\chi(S,\nu)$,
\item[(b)] $\nu(\intt(\widehat E_{i})\cap \{x\in\R^n : h_{i+1}\cdot(x - x_{i+1}) \leq 0\})\ge\frac{1}{2}\nu(\intt(\widehat E_{i}))$ and
\item[(c)] $\intt(\{(x,t): t \geq h_{i+1}\cdot(x - x_{i+1}) - \sum_{k=1}^{i+1}\xi_k)\}) \supseteq \widehat E_{i}\times \{\sum_{k=1}^i\xi_k\}.$
\end{itemize} 
Condition (a) can be ensured by the definition of $\chi(S,\nu)$, (b) can be ensured by choosing one of the closed halfspaces corresponding to the normal minimizing the degeneracy parameter $\chi(S,\nu)$, and (c) can finally be ensured by scaling down $h_{i+1}$ as required. Lastly, define \begin{equation}\label{eq:recursive-epi}\epi(g_{i+1}) = \epi(g_i) \cap \{(x,t): E_0 \times \R: t \geq h_{i+1}\cdot(x-x_{i+1}) - \sum_{k=1}^{i+1}\xi_k \}.\end{equation}
To confirm condition (i), we observe that $E_{i+1} \supseteq \widehat E_{i+1}$, and therefore, 
$$\begin{array}{rcl}\nu(\intt(E_{i+1})) & \geq & \nu(\intt(\widehat E_{i+1})) \\
& = &  \nu(\intt(\tilde E_{i} \cap H_{i+1}))\\
& \ge & \nu(\intt(\widehat E_{i} \cap H_{i+1})) - \epsilon \\
& \ge & \nu(\intt(\widehat E_{i})\cap H_{i+1}) - \chi(S,\nu) - \epsilon \\
& \ge & \frac{1}{2}\nu(\intt(\widehat E_{i}))  - \chi(S,\nu) - \epsilon\\
\end{array}$$
To verify condition (ii), by induction we simply need to verify that if $\mathcal{A}$ queries $x_1, \ldots, x_i$ on the first $i$ iterations while executing on $g_i$, then $g_i(x_k) = g_{i+1}(x_k)$ for all $k = 1, \ldots, i$. This follows from condition (c) above that was maintained during the choice of $h_{i+1}$.
Condition (iii) also follows from condition (c) above that was maintained during the choice of $h_{i+1}$.
Condition (iv) follows from~\eqref{eq:recursive-epi} and the fact that inductively condition (ii) ensures that $\mathcal{A}$ queries the same points $x_1, \ldots, x_i$ on $g_i$ and $g_{i+1}$.
\end{proof}

This concludes the proof of Theorem~\ref{thm:general-optimality}.
\end{proof}

%
%

We obtain, as a special case, a known result of Yudin and Nemirovski~\cite{Nemirovski_Yudin_book} on the optimality of the centerpoint algorithm for continuous convex optimization:

\begin{corollary}\label{cor:continuous-bound}[Continuous convex optimization bounds]
The centerpoint algorithm is optimal amongst cutting plane algorithms for continuous convex optimization in terms of number of function oracle calls, upto the constant factor $\log_2(\frac{e}{e-1})$.
\end{corollary}
\begin{proof} Grunba\"um showed that when $S = \R^n$ and $\nu$ is the Lebesgue measure, $c(S,\nu) = (\frac{n}{n+1})^n \geq \frac{1}{e}$~\cite{Gruenbaum1960}. Theorem~\ref{thm:general-optimality} then gives the result. \end{proof}

\begin{remark} The assumption $c(S,\nu) > 0$ in Theorem~\ref{thm:general-optimality} excludes pathological situations where the algorithm does not terminate. As an example, if $S = \mathcal{S}^{n-1}$ is the unit sphere, $\nu$ is the uniform measure supported on $S$ and $g = \|x\|_2$, then at every iteration of the algorithm, only a single point is excluded and no progress is made in terms of the measure even after countably many iterations.
\end{remark}

\subsection*{A more refined analysis for lattices and mixed-integer lattices}\label{sec:refined-mixed-integer}

In this section we consider the two cases $S=\Z^n$ and $S=\Z^n\times\R^d$ and where $\mu$ is the ``uniform measure'' on a convex set intersected with $S$.
 More precisely, let $K \subseteq \R^n\times \R^d$ be a convex set. Let $\vol_d$ be the $d$-dimensional volume (Lebesgue measure). We define the {\em mixed-integer volume with respect to $K$} as \begin{equation}\label{eq:mixed}\mu_{mixed,K}(C) := \frac{\sum_{z \in \Z^n} \vol_d(C \cap K  \cap (\{z\}\times \R^d))}{\sum_{z \in \Z^n} \vol_d(K  \cap (\{z\}\times \R^d))}\end{equation} for any Lebesgue measureable subset $C\subseteq \R^n \times \R^d$. For later use we want to introduce the notation $\bar\mu_{mixed}(C) = \sum_{z \in \Z^n} \vol_d(C  \cap (\{z\}\times \R^d))$. The dimensions $n$ and $d$ will be clear from the context. 

\begin{remark}\label{rem:minimum}
Let $K\subseteq\R^{n+d}$ be a convex body and let $\mu_{mixed,K}$ denote the mixed-integer volume with respect to $K$, as defined in~\eqref{eq:mixed}. Observe that, if $n=0$, then $\mu_{mixed,K}(H^+(u,x))$ is continuous in $u$. Thus, the infimum over the compact unit sphere is achieved.
When $n\neq0$ the function $\mu_{mixed,K}(H^+(u,x))$ remains continuous nearly everywhere on $\mathcal{S}^{n+d-1}$. Only on $\mathcal{S}^{n-1}\times\{0\}^d$ the function is piece-wise constant
In particular, this implies that the infimum in \eqref{eq:max-value} and \eqref{eq:f} is actually achieved.
\end{remark}

We will show below that when $S = \Z^n$, and $\nu$ is the counting measure on $\Z^n$, $c(S,\nu) = \frac{1}{2^n}$, see Corollary~\ref{cor:helly}. Note that when $S=\Z^n$, one can choose $\delta < 1$ in which case a cutting plane algorithm will return a true optimal solution because if $\nu(\intt(E_{N})) \leq \delta$, this means there is no integer point left in $E_{N}$ and thus $x^\star$ must be an optimal solution. 
To this point we have made no assumption on our initial $E_0$, except that $\nu(E_0)=V > 0$.
It is possible  to design  $E_0$ such that either the lower or the upper bound provided by Theorem~\ref{thm:general-optimality} are best possible. Examples would be $E_0=[0,B]\times\{0\}^{n-1}$ or $E_0=[0,B]\times[0,1]^{n-1}$ respectively, where $B\ge1$.
However, these are rather artificial constructions.
A more common assumption is that an optimal solution $\hat x$ has a bounded representation, say $\|\hat x\|_\infty\le B$ for some natural number $B\geq 1$.
This would imply that we initiate the algorithm with $E_0$ being a box. For ease of presentation we will assume that $E_0=[0,B)^n$, i.e., $E_0$ is not centrally symmetric to the origin.
It follows, that $\nu(E_0)$ equals $B^n$ and $B^{n+d}$ for the integer and the mixed-integer case respectively.
(Of course also other definitions of initial $E_0$'s are plausible, for example balls. In these cases one could also do a more refined analyses as described below. In case of the ball the bound would differ only by a linear factor in terms of the root  of the dimension. This is a consequence of the John's Ellipsoid Theorem.)
Then, the bound in Theorem~\ref{thm:general-optimality} says that the centerpoint based algorithm takes at most $$\left\lceil\log_{\frac{1}{1-c(S,\nu)}}\left(\frac{V}{\delta}\right)\right\rceil  = O \left( n2^n\log_{2}(B)\right)$$
function oracle calls,
where one uses the inequality $-\ln(1-x) \geq x$ for $0 < x < 1$ to deduce that $-\ln(1 - \frac{1}{2^n}) \geq \frac{1}{2^n}$. On the other hand, the lower bound in Theorem~\ref{thm:general-optimality} gives $$\bigg\lceil\log_{2}\left(\frac{V}{\delta}\right)\bigg\rceil -1 = \left\lceil n\log_2(B)\right\rceil -1. $$ This exponential gap between the upper and the lower bounds can be improved using the lattice structure of $S$.

\begin{theorem}\label{thm:pure-integer-bound}[Pure integer convex optimization bound] Let $S = \Z^n$, $\nu$ is the counting measure on $\Z^n$, $E_0=[0,B)^n$ where $B\ge2$ is an integer, and $\delta <1$. Then for any cutting plane algorithm $\mathcal{A}$ 
there exists an instance such that $\mathcal{A}$ makes at least $2^{n-1}\left(\lfloor \log_2(B)\rfloor+1\right)$ first-order evaluation oracle calls on $\hat g$.
 \end{theorem}

\begin{proof}
The proof follows the same idea as in the proof of Theorem~\ref{thm:general-optimality}, except that this time we exploit the discrete structure of $S$. The important thing to illustrate is the choice of the subgradents $\{h_i\}_{i=1}^\infty$ from that proof.

For each $v \in \{0,1\}^{n-1}$ we define the {\em fiber} $F_v := \{0,1,\ldots,B-1\}\times\{v\}$ and let $S_0 := \cup_{v \in \{0,1\}^{n-1}}F_v$. We now construct the adversarial $\hat g$ in an analogous manner as in the proof of Theorem~\ref{thm:general-optimality} by defining the adversarial sub-gradient halfspaces, or cuts. Whenever an algorithm queries the function oracle on a point $x\notin S_0$, then we can always choose the sub-gradient at $x$ such that the halfspace contains $S_0$. 
Otherwise, if $x\in S_0$, we know by definition that $x \in F_{\bar v}$ for some $\bar v \in \{0,1\}$. We now choose the sub-gradient halfspace that removes at most half of the remaining points in $F_{\bar v}$ and keeps the remaining fibers $F_v$, $v \neq \bar v$ intact. It then follows that on each of the $2^{n-1}$ fibers,  
the algorithm has to perform at least 
$\lfloor \log_2( B ) \rfloor + 1$
function oracle calls. Therefore, in all, the algorithm must perform at least $2^{n-1} \left( \lfloor \log_2( B ) \rfloor + 1 \right)$ function oracle calls. 
%
%
\end{proof}

For the mixed-integer case $S = \Z^n\times \R^d$ with the measure $\nu = \bar\mu_{mixed}$, we will show below that $c(S,\nu) \geq \frac{1}{2^n(d+1)}$, see Corollary~\ref{cor:helly}. Similar to the pure integer case above, assuming we start with $E_0$ as the box $[0,B)^n$ for some natural number $B\geq 2$, the bound in Theorem~\ref{thm:general-optimality} says that the centerpoint based algorithm takes at most $$\left\lceil\log_{\frac{1}{1-c(S,\nu)}}\left(\frac{V}{\delta}\right)\right\rceil = O \left( 2^n(d+1)\log_{2}\left(\frac{B^{n+d}}{\delta}\right)\right)$$
function oracle calls.
On the other hand, the lower bound in Theorem~\ref{thm:general-optimality} gives $$\bigg\lceil\log_{2}\left(\frac{V}{\delta}\right)\bigg\rceil = (n+d)\log_2\left(\frac{B}{\delta}\right).$$

However, similar to Theorem~\ref{thm:pure-integer-bound}, one can improve the lower bound in the mixed-integer case too:

\begin{theorem}\label{thm:mixed-integer-bound}[Mixed-integer convex optimization bound] 
Let $S = \Z^n\times\R^n$, and $\nu$ is the mixed integer measure on $S$. Then for any cutting plane algorithm $\mathcal{A}$ 
there exists an instance such that $\mathcal{A}$ makes at least $2^n\left(\log_{2}\left(\frac{B^d}{\delta }\right) +n-1 \right)$ 
first-order evaluation oracle calls on $\hat g$.
 \end{theorem}

\begin{proof}
The proof is similar to the proof of Theorem~\ref{thm:pure-integer-bound}.
We can construct an adversarial function, that treats each fibre $F_v:=\{v\}\times[0,B)$, where $v\in\{0,1\}^{n}$, as separate $d$-dimensional continuous problem.

For all fibers $F_v$, let $\delta_v$ denote the measure of $E_N$ intersected with $F_v$.
By the stopping criteria, it must hold that $\sum_v \delta_v\le\delta$.
By Theorem~\ref{thm:general-optimality} we know that at least $\left(\log_{2}\left(\frac{B^d}{\delta_i }\right)\right) - 1$ function oracle calls must be performed on each $F_v$, and by choosing our sub-gradient halfspaces in the same way as in the proof of Theorem~\ref{thm:pure-integer-bound}, we obtain that the algorithm must make at least
$$N \ge \sum_{v \in \{0,1\}^n} \left(\left(\log_{2}\left(\frac{B^d}{\delta_v }\right)\right) - 1\right)=2^n(\log_2(B^d)-1)-\log_2\left(\prod_{v \in \{0,1\}^n} \delta_v\right)$$ function orcale calls.
Note that the last summand is minimized when $\delta_i=\frac{\delta}{2^n}$ for all $i=1,\ldots,2^n$.
Hence,
$$N \ge 2^n(\log_2(B^d)-1)-\log_2\left(\left(\frac{\delta}{2^n}\right)^{2^n}\right)=2^n\left(\log_{2}\left(\frac{B^d}{\delta }\right) +n-1 \right).$$

%
This completes the proof.
\end{proof}

We finish this section with a few remarks.
As it was already proven by Yudin and Nemirovsky \cite{Nemirovski_Yudin_book}, the centerpoint algorithm is optimal for the continuous case up to a constant factor.
For the pure integer case we could prove that our algorithm is optimal up to a linear factor in $n$.
For the mixed-integer case, if Conjecture~\ref{chap:centerPoints::conj:gruenbaum} would be true, we would have an upper bound of 
$${\bigg\lceil{\log_{{\frac{e}{e-\frac{1}{2^n}}}}\left(\frac{B^{n+d}}{\delta}\right)}\bigg\rceil}.$$
In particular, this would imply that the cutting plane algorithm, using centerpoints, is optimal for mixed-integer optimization, up to a linear factor only in $n$, which would nicely unify the continuous and discrete optimization theory. See Table~\ref{fig:M1}.

\begin{table}
\begin{center}
\small
\begin{tikzpicture}[scale=0.5]

\node at (0,1) {$\;$};

\node at (-1,0) {${S}$};
\node at (4.5,0) {\bf Upper Bound};
\node at (12,0) {\bf Lower Bound}; 

\node at (-1,-2) {${ \R^d}$};
\node at (4.5,-2) { $O\left(\log_{\frac{e}{e-1}}\left(\frac{B^d}{\delta}\right)\right)$ };
\node at (12,-2) {$\Omega\left(\log_{2}\left(\frac{B^d}{\delta}\right)\right)$};
\node at (18.8,-2) {\small{(Yudin \& Nemirovsky)}};

\node at (-1,-4) {${ \Z^n}$};
\node at (4.5,-4) { $O\left( n2^n \log_2(B) \right)$ };
\node at (12,-4) { $\Omega\left(2^{n}\log_2(B)\right)$ };

\node at (-1,-6) {${ \Z^n\times\R^d}$};
\node at (4.5,-6) { $O\left(2^n(d+1)\log_{2}\left(\frac{B^{n+d}}{\delta}\right)\right)$ };
\node at (12,-6) { $\Omega\left(2^n\left(\log_{2}\frac{B^d}{\delta }\right) \right)$ };

\end{tikzpicture}
\end{center}
\caption{Best bounds for the convex optimization problem~\eqref{pro:general-Convex-Opt} with box constraints.} \label{fig:M1}
\end{table}

Next, we want to point out that it is not difficult to generalize the cutting plane algorithm to the constrained optimization case:
\begin{equation*}\label{pro:ConvexMixedIntegerOpt-const}
\min\limits_{x\in\Z^{n}\times\R^d,\atop h(x)\le0} g(x).
\end{equation*}
where $g,h:\R^n\times\R^d\to\R$ are convex functions given by first-order oracles. However, it is crucial that the feasible domain has a reasonable sized measure, as otherwise it might be impossible to find any feasible point, let alone an approximate optimal point. Further, the algorithm can be extended to handle quasi-convex functions, if one has access to separation oracles for their sublevel sets.

Finally, note that for the purely discrete case, when $S=\Z^n$, we can guarantee to find the optimal point of~\eqref{pro:general-Convex-Opt}, provided we choose $\delta<1$.
Only when there are continuous variables, we need to talk about approximations.
Thus, let $S=\Z^n\times\R^d$ with $d\neq 0$.
We assume that for every fixed $x \in \Z^n$, $g(x,y)$ is Lipschitz continuous in the $y$ variables with Lipschitz constant $L$. 
Let $(\hat x,\hat y)\in\Z^n\times\R^{d}$ attain the optimal value $\hat g$ of Problem~\eqref{pro:general-Convex-Opt} and let $(x^\star,y^\star)$ be the best point that the cutting plane algorithm has returned with objective value $g^\star$.
By standard arguments, we can bound  $\bar\mu_{mixed}(E_k)$ from below as follows
\begin{align*}
\bar\mu_{mixed}(E_k) & \ge \bar\mu_{mixed}(\{(x,y)\in\Z^n\times\R^d : g((x,y)) - \hat g \le g^\star - \hat g\})\\
& \ge \bar\mu_{mixed}\left(\left\{ (\hat x,y)\in\{\hat x\}\times\R^d : \|(\hat x,\hat y) -(\hat x,y)\|_2 \le \frac{g^\star -\hat g}{L}\right\}\right) \\
& =\left( \frac{\hat g - g^\star}{L} \right)^d \kappa_d,
\end{align*}
where $\kappa_d$ denotes the volume of the $d$-dimensional unit-ball.
On the other hand, after $N$ iterations it holds that
$$\mu_{mixed}(E_N) \le \delta.$$
Thus, we can guarantee that the algorithm returns a point satisfying $g(x^\star,y^\star)-g(\hat x,\hat y)\le L \left(\frac{\delta}{\kappa_d}\right)^{\frac{1}{d}}$.

\section{Bounds on $\F(S,\mu)$}\label{sec:general-properties}

We first establish some analytic properties of $f_\mu$. This will justify the use of ``maximum'' in~\eqref{eq:max-value}, instead of a supremum. The goal of this section is to establish a bound on the quality of the centerpoints based on Helly numbers of $S$, which will be followed by a better lower bound when $S$ is the mixed-integer lattice.
We will denote the complement of a set $X$ by $X^c$.
We begin with a useful lemma.

\begin{lemma}\label{quasi-concave}
For any probability measure $\mu$, $f_\mu(x)$ defined in~\eqref{eq:f} is quasi-concave on $\R^n$ and upper semicontinuous. 
\end{lemma}

\begin{proof} For quasi-concavity, see Proposition 1 in~\cite{rousseeuw1999depth}, and for upper semicontinuity see Proposition 4 in~\cite{rousseeuw1999depth}. 
\end{proof}

%
%

\begin{remark}\label{rem:maximum}
Lemma~\ref{quasi-concave} 
shows that $\sup_{x\in S} f_\mu(x)$ is always attained. See Proposition 7 in~\cite{rousseeuw1999depth} where this is discussed for $S=\R^n$. The generalization to any nonempty, closed subset $S$ is easy; see also Proposition 5 in~\cite{rousseeuw1999depth} which states the for every $\alpha > 0$, the set
$\{x\in\R^n : f_\mu(x) \ge \alpha \}$ is compact.
\end{remark}

\subsection{A general lower bound based on Helly numbers}\label{sec:general:thm:helly}
We generalize a theorem well-known in the literature on half-space (Tukey) depth~\cite[Proposition 9]{rousseeuw1999depth}; this was earlier stated by Gr\"unbaum \cite[Theorem 1]{Gruenbaum1960} for uniform probability measures on convex bodies. In all of these works, the authors consider $S=\R^n$, as discussed in the introduction. 
{ We consider more general sets $S$.}
For this we define the Helly number of a set $S\subseteq \R^n$.
Let $\mathcal{K}:=\{S \cap K \;|\; K\subseteq\R^n \; \text{convex}\}$.
Then the Helly-number $h=h(S)\in\N$ of $S$ is defined as the smallest number such that the following property is satisfied for all finite subsets $\{C_1,\dots,C_m\}\subseteq\mathcal{K}$:
If
$$C_{i_1}\cap\dots\cap C_{i_h}\neq\emptyset \text{ for all } \{i_1,\dots,i_h\}\subseteq\{1,\dots,m\}$$
then
$$C_1\cap\dots\cap C_m\neq\emptyset.$$
If no such number exists, then $h(S)=\infty$. This extension of Helly's number was first considered by Hoffman~\cite{hoffman1979binding}, and has recently been studied in~\cite{AverkovWeismantel12,averkov2013maximal,de2015helly}. 

\begin{theorem}\label{thm:general:thm:helly} Let $S\subseteq \R^n$ be a nonempty, closed subset and let $\mu$ be such that $\mu(\R^n\setminus S) = 0$.
If $h(S)<\infty$, then $\mathcal{F}(S,\mu)\ge h(S)^{-1}$.
\end{theorem} 


\begin{proof}
The proof follows along similar lines as~\cite[Proposition 9]{rousseeuw1999depth}. It suffices to show that for every $\epsilon > 0$, the set 
$\{x \in \R^n:f_\mu(x)\le h(S)^{-1} -\epsilon\}$ is nonempty. By standard measure-theoretic arguments, there exists a ball $B$ centered at the origin such that $\mu(B) \geq 1 - \frac{\epsilon}{2}$ and 
$\{x \in \R^n:f_\mu(x)\le h(S)^{-1} -\epsilon\} \subseteq B$ (by Remark~\ref{rem:maximum}, 
$\{x \in \R^n:f_\mu(x)\le h(S)^{-1} -\epsilon\}$ is compact). By Proposition 6 in~\cite{rousseeuw1999depth}, 
$$\begin{aligned}&\{x \in S:f_\mu(x)\le h(S)^{-1} -\epsilon\} \\ = & \bigcap\{H \cap S :  H \textrm{ is a closed half space with } \mu(H) \geq 1 - (h(S)^{-1}  -\epsilon)\}.\end{aligned}$$
Define $\mathcal{C} = \{B \cap H \cap S :  H \textrm{ is a closed half space with } \mu(H) \geq 1 - (h(S)^{-1} - \epsilon)\}$. Thus, $\mathcal C$ is a family of compact sets such that 
$\{x \in S:f_\mu(x)\le h(S)^{-1} -\epsilon\} = \bigcap \{C  :  C\in \mathcal{C}\}$. For any subset $\{C_1, \ldots, C_{h(S)}\} \subseteq \mathcal{C}$ of size $h(S)$, we claim $$\mu(C_1^c \cup \ldots C_{h(S)}^c) \leq 1 - h(S)\frac{\epsilon}{2}.$$ This is because each $C_i^c = B^c \cup H_i^c\cup S^c$ for some half space $H_i$ satisfying $\mu(H_i^c) \leq h(S)^{-1} - \epsilon$. Since $\mu(B^c) \leq \frac{\epsilon}{2}$ and $\mu(S^c) = 0$, we obtain that $\mu(C_i^c) \leq h(S)^{-1} -\frac{\epsilon}{2}$. Therefore, $$\mu(C_1 \cap \ldots \cap C_{h(S)}) = 1 - (\mu(C_1^c \cup \ldots C_{h(S)}^c)) \geq 1 - (1 - h(S)\frac{\epsilon}{2})  = h(S)\frac{\epsilon}{2} > 0.$$
This implies that $C_1 \cap \ldots \cap C_{h(S)} \neq \emptyset$. Therefore, by definition of $h(S)$, for every finite subset $\{C_1, \ldots, C_m\} \subseteq \mathcal{C}$, $C_1 \cap \ldots \cap C_m \neq \emptyset$. By the finite intersection property of compact sets, we obtain that 
$\{x \in S:f_\mu(x)\le h(S)^{-1} -\epsilon\} = \bigcap \{C  :  C\in \mathcal{C}\}$ is nonempty.
\end{proof}

By applying the well known bound for the mixed-integer Helly-number \cite{hoffman1979binding,AverkovWeismantel12,de2015helly} we get the following Corollary.
\begin{corollary}\label{cor:helly}
$\mathcal{F}(\Z^n\times\R^d,\mu)\ge\frac{1}{2^n(d+1)}$ for any finite measure $\mu$ on $\R^{n+d}$ such that $\mu(\R^{n+d}\setminus(\Z^n\times\R^d))=0$. In particular, this holds for $\mu_{mixed,K}$ for any convex body $K\subseteq \R^n\times \R^d$.
\end{corollary}

\subsection{Better bounds for the mixed-integer lattice}\label{sec:new-structure}
We would like to improve the bound on $\mathcal{F}(\Z^n\times \R^d,\nu)$ coming from Helly numbers (Theorem~\ref{thm:general:thm:helly} and Corollary~\ref{cor:helly}) when $\nu$ is a mixed-integer measure. Better bounds have been obtained by Gr\"unbraum for the purely continuous case ($n=0$), by exploiting properties of the {\em centroid} of a convex body $K$, which is defined as
$
c_K:=\int_K x \d \mu(x),
$
where the integral is taken with respect to the uniform measure $\mu$ on $K$.
Gr\"unbaum proved in \cite{Gruenbaum1960} that $\mu(H^+(u,c_K))\ge\left(\frac{d}{d+1}\right)^d\ge e^{-1}$ for any $u\in \mathcal{S}^{d-1}$, which immediately implies that $\mathcal{F}(\R^d,\mu)\ge e^{-1}$. This, of course, drastically improves the Helly bound of $\frac{1}{d+1}$.
Note that, even though the centerpoint and centroid are equal for several extreme cases, this is in general not true.
In the following we want to extend these improved bounds to the mixed-integer setting.
Ideally, we would want to prove the following conjecture.

\begin{conjecture}\label{chap:centerPoints::conj:gruenbaum}
Let $S=\Z^n\times\R^d$ and let $\nu:=\mu_{mixed,K}$ for some convex body $K\in\R^{n+d}$.
Then $\mathcal{F}(\Z^n\times\R^d,\nu)\ge \frac{1}{2^n} \left(\frac{d}{d+1}\right)^d \geq \frac{1}{2^n}\frac{1}{e}.$
\end{conjecture}

While we have not been able to resolve the above conjecture, we show that it holds in the regime of convex sets $K$ with ``large'' {\it lattice-width}, where the lattice-width is defined as $$\w(K):=\min_{z\in\Z^n\setminus\{0\}}[\max_{x\in K}u \cdot x-\min_{x\in K}u \cdot x].$$ In fact, we prove something stronger in this regime:

\begin{theorem}\label{thm:boundForLargeLatticeWidth-reform}
There exists a universal constant $\alpha$ such that for all $n, d \in \N$ and any convex body $K \subseteq \R^{n+d}$ with $\omega(K) > 2 c n (n+d)^{5/2} \alpha n^{n+1}$ for some $c \in \R_+$, the following holds:
$$\mathcal{F}(\Z^n\times\R^d,\nu)\ge e^{-\frac{1}{c}-1} + e^{-\frac{2}{c}} -1.$$ 
In particular, when $c \in \R_+$ is such that $e^{-\frac{1}{c}-1} + e^{-\frac{2}{c}} -1 \geq 2^{-n-1}$, we have $\mathcal{F}(\Z^n\times\R^d,\nu) \geq \frac{1}{2^n}\frac12 \geq \frac{1}{2^n} \left(\frac{d}{d+1}\right)^d \geq \frac{1}{2^n}\frac{1}{e}.$
\end{theorem}

We denote the projection of a set $X\subseteq\R^{n+d}$ onto the first $n$ coordinates by $X|_{\R^n}$.

\begin{remark} Theorem~\ref{thm:boundForLargeLatticeWidth-reform} provides some evidence towards our belief in Conjecture~\ref{chap:centerPoints::conj:gruenbaum}. In particular, we see that it holds in two distinct regimes. 

Suppose the convex set $K$ is such that $K|_{\R^n}$ is ``thin'' in every direction; more precisely, suppose there exists a constant $C$ such that for every unit vector $e^i$, $i=1, \ldots, n$ we have $\max_{x\in K} e^i\cdot x - \min_{x\in K}e^i\cdot x \leq C$. Then $\mathcal{F}(\Z^n\times\R^d,\nu)\ge \frac{1
}{C^n}\left(\frac{d}{d+1}\right)^d$. This is witnessed by choosing the centroid of the fiber with at least $\frac{1}{C^n}$ fraction of the total mass of $K \cap (\Z^n \times \R^d)$ -- such a fiber exists because there are at most $C^n$ fibers intersecting $K$.

On the other hand, suppose $K$ is such that $K|_{\R^n}$ is ``fat'' in every direction; more precisely, the hypothesis of Theorem~\ref{thm:boundForLargeLatticeWidth-reform} holds. Then we get an even stronger bound than $\frac{1}{2^n}\left(\frac{d}{d+1}\right)^d$ from Theorem~\ref{thm:boundForLargeLatticeWidth-reform}.
\end{remark}

The rest of this section is devoted to the proof of Theorem~\ref{thm:boundForLargeLatticeWidth-reform}. The main ingredient in the proof of Theorem~\ref{thm:boundForLargeLatticeWidth-reform} is Lemma~\ref{lem:mixedVolume}, where we show that for convex sets with ``large'' lattice width, the $d$-dimensional Lebesgue measure $\bar\nu:= \bar\mu_{mixed}$ of $K\cap(\Z^n\times\R^d)$ can be approximated by the $(d+n)$-dimensional Lebesgue measure $\bar\mu$ of $K$ and vice versa. 
(Note that in this case we do not normalize the measures.)
In the pure integer setting, i.e., $d=0$, this connection is well known.
However, to the best of our knowledge, this kind of result has never been proven for the mixed-integer setting nor explicitly with respect to the lattice width. 

\begin{lemma}\label{lem:mixedVolume}
There exists a universal constant $\alpha$ such that for all $n, d \in \N$ and any convex body $K \subseteq \R^{n+d}$ and $\omega(K|_{\R^n}) \ge c n (n+d)^{5/2} \alpha n^{n+1}$ for some $c \in \R_+$, then the following holds:
\begin{equation*}
e^{-\frac{1}{c}}\le \frac{\bar\nu(K \cap (\Z^n\times\R^d))}{\bar\mu(K)} \le e^{\frac{1}{c}}.
\end{equation*}
\end{lemma}

For the proof of Lemma~\ref{lem:mixedVolume} we need two technical auxiliary lemmata.
The first lemma, Lemma~\ref{chap:centerPoints::lem:centroidHaussdorfdistance}, gives an ellipsoidal approximation of a convex body using the centerpoint as the center of the two ellipsoids used for the approximation. This is a variation on the classical Fritz-John ellipsoidal approximation result.
In the second lemma, Lemma~\ref{sec:app:lem:inscribedBox}, we show that for a convex body $K$ with ``large'' lattice-width, there exists a basis for the mixed integer lattice such that $K$ contains a scaled copy of the fundamental parallelopiped of the lattice with respect to this basis, centered at the centerpoint of $K$.

\begin{lemma}\label{chap:centerPoints::lem:centroidHaussdorfdistance}
Let $K\subseteq\R^n$ be a compact convex set with nonempty interior and let $\mu$ be the uniform measure with respect to $K$. Further let $x^\star\in\mathcal{C}(\R^n,\mu)$.
Then, there exists an ellipsoid $E$ centered at the origin such that
$$x^\star+E\subseteq K\subseteq x^\star+n^{5/2}E.$$
\end{lemma}

\begin{proof}
Without loss of generality we assume that $x^\star=0$.
We prove that for any $u\in S^{n-1}$ 
\begin{equation}\label{chap:centerPoints::equ:ratioCenterPoints}
\frac{1}{n^2}\le\left|\frac{\max_{x\in K}u^\T x}{\min_{x\in K}u^\T x}\right|\le n^2.
\end{equation}

It suffices to show the lower bound, since the upper bound follows from replacing $u$ with $-u$.


Let $\left|\frac{\max_{x\in K}u^\T x}{\min_{x\in K}u^\T x}\right|$ be minimized at $u = \bar u$. Since the arguments below are invariant to scaling, we will assume that 
$\min_{x\in K}\bar u^\T x=-1$ and $\max_{x\in K}\bar u^\T x=:\alpha$, and assume to the contrary that
that $\alpha<\frac{1}{n^2}$.
Let $z:=\argmin_{x\in K}\bar u^\T x$. 
We define for every $t\in\R$ the set $K_t:=K\cap\{x\in\R^n : u^\T x=t\}$.
Further, we define $C:=z+\cone(K_0-z)$, $X_1:=\{x\in\R^n : -1\le u^\T x \le0 \}$ and $X_2=\{x\in\R^n : 0\le u^\T x \le \alpha \}$.
Then $K\cap X_1\supset C\cap X_1\;\text{ and }\;K\cap X_2\subseteq C\cap X_2.$
By Gr\"unbaum's theorem~\cite[Theorem 2]{Gruenbaum1960} we have that $1-\left(\frac{n}{n+1}\right)^n\ge\mu(K\cap X_1)\ge\mu(C\cap X_1)=\frac{1}{n}\frac{V}{\vol(K)}$, where $V$ represents the $(n-1)$-dimensional measure of $K_0$.
On the other hand we have $\left(\frac{n}{n+1}\right)^n\le\mu(K\cap X_2)\le\mu(C\cap X_2)=\big(\frac{(1+\alpha)^n}{n}-\frac{1}{n}\big)\frac{V}{\vol(K)}$.
Combining these two inequalities, we arrive at the inequality
$$\frac{1}{e-1} \leq \frac{(\frac{n}{n+1})^n}{1 -(\frac{n}{n+1})^n} \leq (1+\alpha)^n - 1 \leq \bigg(1 + \frac{1}{n^2}\bigg)^n - 1$$
However, $\bigg(1 + \frac{1}{n^2}\bigg)^n - 1 < \frac{1}{e-1}$ for all $n \geq 2$, leading to a contradiction. 

We now define $$sym(K) = \max\{\alpha \geq 0 : \alpha(-y) \in K \textrm{ for every } y \in K\},$$ a notion that was introduced by Minkowski, and has been extensively studied in convex geometry literature~\cite{belloni2008symmetry,toth2015measures}. Combined with Proposition 1 in~\cite{belloni2008symmetry}, equation~\eqref{chap:centerPoints::equ:ratioCenterPoints} implies that $sym(K) \geq \frac{1}{n^2}$. Then Theorem 7 in~\cite{belloni2008symmetry} shows that there exists an ellipsoid $E$ centered at the origin satisfying $E\subseteq K\subseteq n^{5/2}E.$
%
\end{proof}

\begin{lemma}\label{sec:app:lem:inscribedBox}
There exists a universal constant $\alpha$ such that the following holds for all $n, d \in \N$. Let $K\subseteq\R^{n+d}$ be a convex body and let $x^\star\in\mathcal{C}(\R^n,\mu)$, where $\mu$ is the uniform measure with respect to $K$. If $\w(K|_{\R^n})\ge c n (n+d)^{5/2} \alpha n^{n+1}$ for some $c\in\R_+$,
then there exists a matrix $B=[b_1,\dots,b_{n}]\in\R^{(n+d)\times n}$ such that
$$x^\star + cn B [-1/2,1/2]^{n} \subseteq K.$$
and
$b_1|_{\R^n},\dots,b_{n}|_{\R^n}$ forms a lattice basis of $\Z^n$.
\end{lemma}

\begin{proof}
By Lemma~\ref{chap:centerPoints::lem:centroidHaussdorfdistance}, there exists an ellipsoid $E$ such that $x^\star + E\subseteq K\subseteq x^\star + (n+d)^{5/2}E$.
We define $\phi:\R^n\to\R^n$ as the linear map such that $\phi(E|_{\R^n})=\{x\in\R^n : \|x\|_2\le1\}.$
Let 
$\Lambda:=\phi(\Z^n)$.

Let $\widehat B$ be a matrix whose columns $\widehat B_{\star,1}, \ldots, \widehat B_{\star,n}$ form a Korkine-Zolotarev basis of $\Lambda$ \cite{KorkineZolotareff}.
Then, a well known property is that
\begin{equation}\label{chap:centerPoints::eq:KorkinZolotarev}
\|\widehat B_{\star,1}\|_2\cdots\|\widehat B_{\star,n}\|_2\le \alpha n^n\det(\Lambda)
\end{equation}
(see \cite[Theorem 2.3]{LagariasLenstraSchnorr}),
where $\alpha $ is a universal constant. 
Further, the Gram-Schmidt orthogonalization $\tilde{B}_{\star,1},\dots,\tilde{B}_{\star,n}$ of $\widehat B$ satisfies \begin{equation}\label{eq:gram-schmidt}\|\tilde{B}_{\star,i}\|_2\le\|\widehat B_{\star,i}\|_2 \textrm{ for all }i=1,\dots,n\end{equation} and it holds that \begin{equation}\label{eq:det}\det(\Lambda)=\det(\widehat B)=\det(\tilde{B})=\prod_{i=1}^n\|\tilde{B}_{\star,i}\|_2\end{equation} (see, for example, \cite[Chapter 28]{Gruber-Book07}).
Since $K\subseteq x^\star + (n+d)^{5/2} E$, the definition of the lattice-width $\w(K|_{\R^n})$ implies that \begin{equation}\label{eq:width-bound}\|\tilde{B}_{\star,n}\|_2\le \frac{2(n+d)^{5/2}}{\w(K|_{\R^n})}
\end{equation} Putting all these relations together, for all $j \in \{1, \ldots, n\}$, we have the following:
$$\begin{array}{rcll}
(\prod_{i \neq n} \|\tilde B_{\star,i}\|) \cdot\| \widehat B_{\star,n}\| &\leq &(\prod_{i \neq n} \|\widehat B_{\star,i}\|)\cdot \| \widehat B_{\star,n}\| &\textrm{using~\eqref{eq:gram-schmidt}} \\
& \leq & \alpha n^n\det(\Lambda)& \textrm{using~\eqref{chap:centerPoints::eq:KorkinZolotarev}} \\
& = & \alpha n^n(\prod_{i=1}^n \|\tilde B_{\star,i}\|)& \textrm{using~\eqref{eq:det}} \\
& \leq & (\prod_{i \neq n} \|\tilde B_{\star,i}\|)\alpha n^n \frac{2(n+d)^{5/2}}{\w(K|_{\R^n})} & \textrm{using~\eqref{eq:width-bound}} \\
\end{array}
$$ and thus we obtain
$$\|\widehat{B}_{\star,n}\|_2\le \frac{2\alpha n^n(n+d)^{5/2}}{\w(K|_{\R^n})}.
$$
If we change the order of the columns in $\widehat B$, the equations~\eqref{chap:centerPoints::eq:KorkinZolotarev}-\eqref{eq:width-bound} still hold (with a different $\widehat B_{\star, n}$), and thus 
we obtain a bound on the Euclidean length of all Korkine-Zolotarev vectors, i.e., for all $i=1,\dots,n$
$$\|\widehat B_{\star,i}\|_2\le\frac{2\alpha n^n(n+d)^{5/2}}{\w(K|_{\R^n})}.$$
This implies that 
\begin{equation*}\label{lem:mixedVolume::equ:2}
\frac{1}{n}\frac{1}{\alpha n^n} \frac{\w(K|_{\R^n})}{(n+d)^{5/2}} \widehat B [-1/2,1/2]^n\subseteq \{x\in\R^n : \|x\|_2\le1\} \\
\end{equation*}
Since $\w(K|_{\R^n}) \geq c n (n+d)^{5/2}\alpha n^{n+1}$, we obtain that $$cn\widehat B [-1/2,1/2]^n \subseteq \{x\in\R^n : \|x\|_2\le1\}.$$

We now use the fact that there exists an affine subspace $H \subseteq \R^{n+d}$ of dimension $n$ such that $E|_{\R^n} = (E \cap H)|_{\R^n}$. This follows from a repeated application of Lemma 3 from~\cite{del2016ellipsoidal}. Let $B$ be the matrix whose columns span the linear space parallel to $H$ and these columns project to $\phi^{-1}(\widehat B_{\star,1}), \ldots, \phi^{-1}(\widehat B_{\star,n})$. $B$ now satisfies the condition desired.
\end{proof}

We are now ready to prove Lemma~\ref{lem:mixedVolume}. 

\begin{proof}[Proof of Lemma~\ref{lem:mixedVolume}]
By Lemma~\ref{sec:app:lem:inscribedBox} there exists a matrix $B\in\R^{(n+d)\times n}$ such that $x^\star + cn\, B [-1/2,1/2]^{n} \subseteq K$.
Since $B|_{\R^n}$ is unimodular, we may assume after an unimodular transformation that $B|_{\R^n}$ equals the identity matrix.
After a further volume preserving linear transformation we may even assume that $B$ equals the first $n$ unit vectors in $\R^{n+d}$.
Since $K$ is full dimensional, there exists an $\epsilon>0$ such that $c\,n[-1/2,1/2]^n\times \epsilon[-1/2,1/2]^d\subseteq K - x^\star$. To make the calculations below easy, we translate everything by $-x^\star$, so that we assume below that $x^\star = 0$.

We next exploit that $\lim_{k \to \infty} \frac{1}{k^d} |K \cap (\Z^{n}\times\frac{1}{k}\Z^d)|=\bar\nu(K \cap (\Z^n\times\R^d))$, where $|\cdot|$ denotes the cardinality. 
We will use this fact to prove the following 
\begin{claim}\label{claim:discretize}
\begin{equation*}
\left( 1- \frac{1}{cn} \right)^n \left( 1- \frac{1}{\epsilon k} \right)^d \le \frac{\frac{1}{k^d}|K \cap (\Z^{n}\times\frac{1}{k}\Z^d)|}{ \bar\mu (K)} \le \left( 1+ \frac{1}{cn} \right)^n \left( 1+ \frac{1}{\epsilon k} \right)^d.
\end{equation*}
\end{claim}
\begin{proof}
Let $Q:=[-1/2,1/2]^{n}\times[-1/2k,1/2k]^d$.
Further we define 
$\bar{K} := (K \cap (\Z^{n}\times\frac{1}{k}\Z^d)) + Q$ and
$K_{\lambda,\gamma} :=  \left(\begin{smallmatrix} \lambda I_n & 0  \\ 0 & \gamma I_d \end{smallmatrix}\right) K$ for $\lambda,\gamma \in \R_+$. Notice that $K_{\lambda_1,\gamma_1} + K_{\lambda_2,\gamma_2} = K_{\lambda_1 + \lambda_2, \gamma_1 + \gamma_2}$ and $Q \subseteq K_{\frac{1}{cn},\frac{1}{\epsilon k}}$.

Since $\frac{1}{k^d}|K \cap (\Z^{n}\times\frac{1}{k}\Z^d)| = \bar\mu (\bar{K})$, it suffices to show that 
$$ K_{1-\frac{1}{cn},1-\frac{1}{\epsilon k}} \subseteq \bar K \subseteq K_{1+\frac{1}{cn},1+\frac{1}{\epsilon k}}$$
One of the containments follows from the observations that $\bar K \subseteq K + Q$ and $K + Q \subseteq K_{1+\frac{1}{c}n,1+1/\epsilon k}$.
In order to prove $K_{1-\frac{1}{cn},1-\frac{1}{\epsilon k}} \subseteq \bar K$, suppose to the contrary that there exists an $x \in K_{1-\frac{1}{cn},1-\frac{1}{\epsilon k}} \setminus \bar{K}$.
We define $z\in\Z^n\times\frac{1}{k}\Z^d$, such that $z_i=\lfloor x_i \rceil$ for $1 \le i \le n$ and $z_i=\frac{1}{k} \lfloor k\, x_i \rceil$ for $n < i \leq n+d$.\footnote{For $x\in\R$, $\lfloor x \rceil$ denotes the integer $z\in\Z$ such that for each component $|z-x|\le\frac{1}{2}$.}
Then, since $z - x \in Q \subseteq K_{\frac{1}{cn},\frac{1}{\epsilon k}}$ and $x \in K_{1-\frac{1}{cn},1-\frac{1}{\epsilon k}}$, we obtain that $z$ must be in $K \cap (\Z^n\times\frac{1}{k}\Z^d)$. Since $Q$ is symmetric, $x - z$ is also in $Q$, and therefore $x \in z + Q \subseteq K \cap (\Z^n\times\frac{1}{k}\Z^d) + Q = \bar K$, which contradicts the choice of $x$.
\end{proof}

We get the desired inequalities from Claim~\ref{claim:discretize} by letting $k$ go to infinity and using the fact that $(1 + \frac{1}{cn}) ^n \leq e^{\frac{1}{c}}$ and $(1 - \frac{1}{cn}) ^n \geq e^{-\frac{1}{c}}$ for all $c \in \R_+$ and $n \in \N$.
\end{proof}

To complete the proof of Theorem~\ref{thm:boundForLargeLatticeWidth-reform} we introduce the following technical rounding procedure.
Let $K$ be a convex body with a sufficiently large lattice width, i.e., $\w(K)>c n (n+d)^{5/2} \alpha n^n \sqrt{n}$ for some positive integer $c$, where $\alpha$ is the constant from Lemma~\ref{lem:mixedVolume}. Let $\mu$ be the uniform measure on $K$ and let $x^\star\in\mathcal{C}(\R^{n+d},\mu)$. By Lemma~\ref{sec:app:lem:inscribedBox}, there exist $b_i\in (-x^\star+K)\cap(\Z^n\times\R^d)$ for $i=1,\dots,n$ such that $b_1|_{\R^n},\dots,b_{n}|_{\R^n}$ is a lattice basis of $\Z^n$ and $x^\star + cn B [-1/2,1/2]^{n} \subseteq K$.
In addition we define $b_i\in\R^{n+d}$ as the $i$-th unit vector for $i=n+1,\dots,n+d$.
Hence, $b_1,\dots,b_{n+d}$ define a basis of $\R^{n+d}$.

Given $x=\sum_{i=1}^{n+d}\lambda_i b_i\in\R^{n+d}$ with $\lambda_i\in\R$ for all $i$,
we define $[ x ]_K\in\Z^n\times\R^d$ as $\sum_{i=1}^n \lfloor \lambda_i \rceil b_i +\sum_{i=n+1}^{n+d}  \lambda_i b_i$, i.e., we round $x$ to a close mixed-integer point with respect to $K$ (the dependence on $K$ is through Lemma~\ref{sec:app:lem:inscribedBox} which defines the matrix $B$).

\begin{theorem}\label{thm:boundForLargeLatticeWidth} Let $\alpha$ be the constant from Lemma~\ref{lem:mixedVolume}. Let $\nu:=\mu_{mixed,K}$, where  $K \subseteq \R^{n+d}$ is a  convex body and $\nu(K)\neq 0$, and let $x^\star$ be the centerpoint with respect to $\mu$, the uniform measure on $K$.
If $\omega(K|_{\R^n}) > 2 c n (n+d)^{5/2} \alpha n^{n+1}$ for some $c \in \R_+$, then
$$f_{\nu}([ x^\star ]_K)\ge e^{-\frac{1}{c}}\mathcal{F}(\R^{d+n},\mu) + e^{-\frac{2}{c}} -1.$$ 
\end{theorem}
Gr\"unbaum's Theorem implies then, that $\mathcal{F}(\Z^n\times\R^d,\nu)\ge e^{-\frac{1}{c}-1} + e^{-\frac{2}{c}} -1$, giving us Theorem~\ref{thm:boundForLargeLatticeWidth-reform}.
\begin{proof}[Proof of Theorem~\ref{thm:boundForLargeLatticeWidth}] As before, let $\bar \mu$ denote the $(d+n)$-dimensional Lebesgue measure with respect to $K$ and let $\bar \nu$ denote the $d$-dimensional Lebesgue measure with respect to $K\cap(\Z^n\times\R^d)$, i.e. they are not normalized.

In a first step we prove the following claim:
\begin{claim}\label{claim:half-space-approx} For any half-space $H^+$, $$\nu(H^+) \geq \mu(H^+) + e^{-\frac{2}{c}} - 1.$$\end{claim}

\begin{proof}[Proof of Claim~\ref{claim:half-space-approx}]
Let $H^+$ be any half-space 
and let $H^-$ denote its closed complement.
The lattice-width of either $(K\cap H^+)|_{\R^n}$ or $(K\cap H^-)|_{\R^n}$ is larger or equal than $\omega(K|_{\R^n})/2$.
If $\w(K\cap H^-)\ge  c n (n+d)^{5/2} \alpha n^n \sqrt{n}$, then, by Lemma~\ref{lem:mixedVolume},
\begin{align*}
\nu(H^+)=\frac{\bar \nu(K \cap H^+)}{\bar \nu(K)}&\ge \frac{e^{-\frac{1}{c}}\bar\mu(K) - e^{\frac{1}{c}}\bar\mu(K\cap H^-)}{e^{\frac{1}{c}}\bar\mu(K)}\\
&= \frac{\bar \mu(K\cap H^+)}{\bar\mu(K)}+\frac{(e^{-\frac{1}{c}}-e^{\frac{1}{c}})\bar\mu(K)}{e^{\frac{1}{c}}\bar\mu(K)}\\
&= \mu(H^+) + e^{-\frac{2}{c}} - 1.
\end{align*}
If $\w(K\cap H^+)\ge  c n (n+d) \alpha n^n \sqrt{n}$, then, by Lemma~\ref{lem:mixedVolume},
\begin{align*}
\nu(H^+)=\frac{\bar \nu(K \cap H^+)}{\bar \nu(K)}&\ge \frac{e^{-\frac{1}{c}}\bar\mu(K\cap H^+)}{e^{\frac{1}{c}}\bar\mu(K)}\\
&= e^{-\frac{2}{c}}\mu(H^+) \\
& = \mu(H^+) - (1 - e^{-\frac2c})\mu(H^+) \\
& \ge \mu(H^+) + e^{-\frac2c} + 1. 
\end{align*}
The last inequality holds since $\mu(H^+) \le 1$ and $1 - e^{-\frac2c} \geq 0$.
\end{proof}

In the second step we bound the error made by rounding the $x^\star$ to $[ x^\star ]_K$. Note that this is done with respect to a matrix $B$ that is returned from Lemma~\ref{sec:app:lem:inscribedBox}. We first make a unimodular transformation so that $B|_{\R^n}$ is the identity and then make an affine transformation so that $B$ consists of the first $n$ unit vectors in $\R^{n+d}$. To make a further simplification, we translate everything by $-x^\star$ so that $x^\star = 0$.

Define $K_{\lambda, \gamma} := \left(\begin{smallmatrix} \lambda I_n & 0  \\ 0 & \gamma I_d \end{smallmatrix}\right) K$ for any $\lambda, \gamma \geq 0$. Since we assume $x^\star = 0$, our rounding procedure implies that $[x^\star] \in B[-1/2,1/2]^n$. By Lemma~\ref{sec:app:lem:inscribedBox}, this implies that $[x^\star] \in K_{\frac{1}{cn},0}$. Therefore, $[x^\star] + K_{1 - \frac{1}{cn},1} \subseteq K$. This implies that for any $u\in \mathcal{S}^{n+d-1}$, $\mu(H^+(u,[ x^\star ]_K)) \ge (1 - \frac{1}{cn})^n \mu(H^+(u,x^\star))  \geq e^{-\frac{1}{c}} \mu(H^+(u,x^\star))$.

Together with the previous claim it follows that
$$\begin{array}{rcl} f_{\nu}([ x^\star ]_K) & = & \max_{u\in \mathcal{S}^{n+d-1}} \nu(H^+(u,[ x^\star ]_K)) \\
& \geq & \max_{u\in \mathcal{S}^{n+d-1}} \mu(H^+(u,[ x^\star ]_K)) + e^{-\frac{2}{c}} - 1 \\
& \geq & \max_{u\in \mathcal{S}^{n+d-1}} e^{-\frac{1}{c}}\mu(H^+(u,x^\star)) + e^{-\frac{2}{c}} - 1 \\
&\ge &e^{-\frac{1}{c}}\mathcal{F}(\R^{d+n},\mu) + e^{-\frac{2}{c}} -1.\end{array}$$ 
This completes the proof.
\end{proof}

\section{Computational Aspects}\label{sec:algorithms}

All our algorithms discussed in this section are under the standard Turing machine model of computation. We say that $x \in S$ is an $\epsilon$-centerpoint for $S, \mu$, if $f_\mu(x) \geq \F(S,\mu) - \epsilon$ where $\F(S,\mu)$ is defined in~\eqref{eq:max-value} and $f_\mu$ is defined in~\eqref{eq:f}. 

A central tool in the following algorithms for computing (approximate) centerpoints is solving convex mixed-integer optimization problems. The classical result here is due to Gr\"otschel et al.~\cite{GroetschelLovaszSchrijver-Book88}. This classical algorithm requires an access to a first-order oracle for the convex function at all points in $\R^n$. It can be modified to solve the problem with access to a first-order oracle that only queries mixed-integer points (as opposed to any point in $\R^n$). This modification will be useful for us in this section, in particular, for Theorem~\ref{chap:centerPoints::thm:2dAlg}. For completeness, we give a description of the result most amenable for our purposes; an appropriate reference for this version is~\cite{OertelWagnerWeismantel14}.

\begin{theorem}\label{thm:timm-conv-min}
Let $S = \Z^n\times \R^d$, $B \ge 0$ and $\epsilon>0$. Let $f:\R^n \times \R^d \to \R$ be a quasi-concave function equipped with an oracle such that there exists $\delta \leq (C\frac{\epsilon}{B})^{n+d}$ for some universal constant $C$ independent of $B,n,d,\epsilon$ with the following property. For any point $(\bar x, \bar y) \in S$, the oracle returns an approximate function value $\bar f$ and an approximate separation vector $\bar u \in \mathcal{S}^{n+d-1}$ with the following guarantees:
\begin{itemize}
\item[(i)] There exists an optimal solution in $\argmax_{x\in S}f(x) $ with norm at most $B$,
\item[(ii)] $|f(\bar x, \bar y) - \bar f| \leq \delta$,
\item[(iii)] $\| u - \bar u\|_\infty \leq \delta$ for some $u \in \mathcal{S}^{n+d-1}$ satisfying  $\{(x,y)\in S : f(x,y)\ge f(\bar x, \bar y)\} \subseteq \{(x,y): u\cdot(x,y) \leq u\cdot(\bar x, \bar y)\}$.
\end{itemize}
Then there is an algorithm that computes a point $x^* \in S$ such that $ \max_{x\in S}f(x) - f(x^*) \leq \epsilon$. Moreover, if $n$ is fixed, the algorithm runs in time polynomial in $\log(B)$, $\log(\frac{1}{\epsilon})$ and $d$.
\end{theorem}


\subsection{Exact Algorithms}\label{sec:exact} \subsubsection{Uniform measure on polytopes}\label{subsubsection:UniformMeasureOnPolytopes}

Since the rationality of the centerpoint for uniform measures on rational polytopes is an open question, we consider an ``exact'' algorithm as one which returns an $\epsilon$-centerpoint and runs in time polynomial in $\log(\frac1\epsilon)$ and the size of the description of the rational polytope.

\begin{theorem}\label{thm:exact-uniform-algo}
Let $n$ be a fixed natural number. 
There is an algorithm which takes as input a rational polytope $P \subseteq \R^n$ and $\epsilon > 0$, and returns an $\epsilon$-centerpoint for $S=\R^n$ and $\mu$, the uniform measure on $P$. The algorithms runs in time polynomial in the size of an irredundant description of $P$ and $\log(\frac1\epsilon)$.
\end{theorem}
%
%

\begin{proof} Since  $f_\mu$ defined in~\eqref{eq:f} is quasi-concave by Lemma~\ref{quasi-concave}, an $x^*$ satisfying $f_\mu(x^*) \geq \F(S,\mu) - \epsilon$ {  can be found using Theorem~\ref{thm:timm-conv-min},} if one has an approximate evaluation oracle for $f_\mu$, and an approximate separation oracle for the level sets. 

Implementing these oracles boils down to the following: Given $\bar x\in \R^n$ and $\delta > 0$, find $\bar u\in \mathcal{S}^{n-1}$ such that \begin{equation}\label{eq:min-u}\mu(H^+(\bar u, \bar x)) \leq\min_{u \in \mathcal{S}^{n-1}}\mu(H^+(u,\bar x)) + \delta \;\;\text{ and }\;\; \| u - \bar u\|_\infty \leq \delta,\end{equation}
 for some $u \in \argmin_{u \in \mathcal{S}^{n-1}}\mu(H^+(u,\bar x))$.



Given $\bar x$, let $\mathcal{P}$ be the set of all partitions of the vertices of $P$ into two sets that can be achieved by a hyperplane through $\bar x$. (Note that, since $n$ is assumed to be fixed, the number of vertices of $P$ is polynomially bounded and they can be computed in time bounded by a polynomial in terms of the input-size of $P$.) This induces a covering of the sphere $\mathcal{S}^{n-1}$: For each $X \in \mathcal{P}$ define $U_X$ to be the set of $u \in \mathcal{S}^{n-1}$ such that the hyperplane $u\cdot x = u\cdot \bar x$ induces the partition $X$ on the vertices of $P$. The number of such partitions is closely related to the VC-dimension of hyperplanes, and in particular, is easily seen to be $O(M^n)$ where $M$ is the number of vertices of $P$. Indeed, let $X\in\mathcal{P}$. Then there exists an $u\in\cl(U_X)$ (where $\cl(\cdot)$ denotes the closure), such that the hyperplane defined by $u$ contains $n-1$ vertices $\{v_1, \ldots, v_{n-1}\}$ of $P$ such that $\{\bar x, v_1, \ldots, v_{n-1}\}$ are affinely independent. Thus, we can construct an $u_X \in U_X$ by perturbing this hyperplane to obtain the partition $X$. Moreover, one can enumerate these partitions in the same amount of time, by picking $n-1$ vertices $\{v_1, \ldots, v_{n-1}\}$ of $P$ such that $\{\bar x, v_1, \ldots, v_{n-1}\}$ are affinely independent.

To solve problem~\eqref{eq:min-u}, we will proceed along these steps.

\begin{enumerate}
\item For each $X \in \mathcal{P}$, find $\bar u_X\in \mathcal{S}^{n-1}$ be such that $$\mu(H^+(\bar u_X, \bar x)) \leq\min_{u \in U_X}\mu(H^+(u,\bar x)) + \delta \;\;\text{ and }\;\; \| u - \bar u_X\|_\infty \leq \delta, $$
 for some $u \in \argmin_{u \in U_X}\mu(H^+(u,\bar x))$.
\item Pick $X^*$ such that $ \mu(H^+(\bar u_{X^*},\bar x)) \leq  \mu(H^+(\bar u_X,\bar x))$ for all $X \in \mathcal{P}$ and report $\bar u_{X^*}$ as the solution to~\eqref{eq:min-u}.
\end{enumerate}

To complete the proof, we need to implement Step 1. above in polynomial time. This is done in Lemma~\ref{lem:compute-uX}. \end{proof}
\begin{lemma}\label{lem:compute-uX}
For a fixed $X \in \mathcal{P}$, one can compute $\bar u_X\in \mathcal{S}^{n-1}$ such that $$\mu(H^+(\bar u_X, \bar x)) \leq\min_{u \in U_X}\mu(H^+(u,\bar x)) + \delta\;\;\text{ and }\;\; \| u - \bar u_X\|_\infty \leq \delta, $$
 for some $u \in \argmin_{u \in U_X}\mu(H^+(u,\bar x))$,
 using an algorithm whose running time is bounded by a polynomial in $\log(\frac1\delta)$ and the size of an irredundant description of $P$.
\end{lemma}
This lemma can be proved using methods from real algebraic geometry for quantifier elimination in systems of polynomials inequalities.

\begin{proof} 
For a fixed partition $X\in \mathcal{P}$ the feasible region $U_X$ is described by a system of linear inequalities $A u \le b$ and a single quadratic equality  constraint $u_1^2 + u_2^2 + \ldots + u_n^2 = 1$. We claim the objective function can be written as the ratio of two polynomials in $u_1, \ldots, u_n$. Subject to these constraints, we need to minimize $\mu(H^+(u,\bar x))$. Since $X$ is the partition of the vertices of $P$ induced by the hyperplane $u\cdot x = u\cdot \bar x$ (since $u \in U_X$), the set of edges intersected by this hyperplane is fixed. Moreover, the coordinates of the point of intersection of any such edge and this hyperplane can be expressed by a ratio of linear functions of $u_1, \ldots, u_n$. Indeed, suppose the edge intersected is the convex hull of the vertices $v_1, v_2 \in \R^n$. Then there exists $\lambda \in [0,1]$ such that $u\cdot(\lambda v_1 + (1-\lambda)v_2) = u\cdot \bar x$. Thus, $\lambda = \frac{u\cdot(\bar x - v_2)}{u\cdot(v_2 - v_1)}$, and the point of intersection is $\lambda v_1 + (1-\lambda)v_2$ which is a ratio of linear functions of $u$. Also, $P\cap H^+(u,\bar x)$ can be decomposed into a simplicial complex whose combinatorial structure only depends on $X$ and not on the actual values of $u \in U_X$. The volume of $P\cap H^+(u,\bar x))$ can be written as the sum of the volumes of these simplices. Since the volume of a simplex can be written as a polynomial in the coordinates of its vertices, we obtain that $\mu(H^+(u,\bar x))$ is sum of ratios of polynomials in $u_1, \ldots, u_n$ with degree bounded by a function of $n$ only, which can be written as a single ratio of polynomials in $u_1, \ldots, u_n$ where the degrees of the polynomials are bounded by a function of $n$ only. Thus, finding $u_X \in \argmin_{u \in U_X}\mu(H^+(u,\bar x))$ is equivalent to solving a mathematical optimization problem of the following type:

$$
\min_{u_1, \ldots, u_n} \frac{p(u)}{q(u)} \qquad \textrm{s.t. }\; A u \leq b, \;\; u_1^2 + u_2^2 + \ldots + u_n^2 = 1.
$$

The above is equivalent to the following polynomial optimization problem:

$$
\min_{z, u_1, \ldots u_n} z \qquad \textrm{s.t. }\;p(u) = z\cdot q(u), \;\; Au \leq b, \;\; u_1^2 + u_2^2 + \ldots + u_n^2 = 1.
$$

This optimization problem can be solved to within $\delta$ accuracy of the objective and the solution by performing a binary search on the objective value and  using quantifier elimination methods for testing feasibility of polynomial systems of inequalities and equalities. For polynomial systems with a fixed number of variables this can be done in polynomial time in the size of the coefficients~\cite{basu1996combinatorial}. See also ``Remark'' on page 2 of~\cite{grigor1988solving}.
\end{proof}

\subsubsection{Counting measure on the integer points in two dimensional polytopes}\label{chap:centerPoints::subsec:appAlgFor2Dim}
If we use the counting measure on the integer points in a polytope, the algorithm requires no accuracy parameter $\epsilon$.


\begin{theorem}\label{chap:centerPoints::thm:2dAlg}
Let $P=\{ x\in\R^2 : Ax\le b \}$ be a rational polytope, where $A\in\Z^{m \times 2}$ and $b \in \Z^m$, such that $P\cap\Z^2\neq\emptyset$.
Let $\mu$ denote the uniform measure on $P\cap\Z^2$.
Then in polynomial time in the input-size of $A$ and $b$, one can compute a point 
$$z\in\mathcal{C}(\Z^2,\mu).$$
\end{theorem}

\begin{proof}
By utilizing the fact that $f_\mu$ is concave (Lemma~\ref{quasi-concave}) and Theorem~\ref{thm:timm-conv-min}, it suffices to show that for a given $\bar x\in\Z^2$ one can compute in polynomial time
$$\bar u \in \argmin_{u \in \mathcal{S}^{1}}\mu(H^+(u,\bar x)).$$

Let $g: [0,2\pi) \to [0,1]$ be defined as $g(\alpha):=\mu(H^+((\sin(\alpha),\cos(\alpha))^\T,\bar x))$.
The key observations are that $g$ is piecewise constant and that the domain $[0,2\pi)$ can be partitioned into a polynomial number of intervals $S_i$ such that $g$ is monotone on each of them.
This implies, that in order to compute $\bar u$, one only needs to evaluate $g$ at the beginning and the end of each interval $S_i$.

Let $l^+(\alpha)$ denote the line segment $P\cap\{\bar x + \lambda (\sin(\alpha+\pi/2),\cos(\alpha+\pi/2))^\T : \lambda\ge0\}$ and let $l^-(\alpha)$ denote $P\cap\{\bar x + \lambda (\sin(\alpha-\pi/2),\cos(\alpha-\pi/2))^\T : \lambda\ge0\}$.
Observe that $g(\alpha)$ is monotone increasing if the line segment $l^+(\alpha)$ is longer than the line segment $l^-(\alpha)$ and $g(\alpha)$ is monotone decreasing if the line segment $l^+(\alpha)$ is shorter than the line segment $l^-(\alpha)$.
Hence, the monotonicity can only change when the two lengths are equal.
All those critical $\alpha$ can be computed by comparing each pair of facets.
\end{proof}

\subsection{Approximation algorithms}\label{chap:centerPoints::subsec:appAlgForFixDim}

\subsubsection{A Lenstra-type algorithm to compute approximate centerpoints} 


As we already pointed out in Section~\ref{sec:center-opti}, centerpoints can be used to design ``optimal'' oracle-based algorithms for convex mixed-integer optimization problems.
In turn, it is possible to employ linear mixed-integer optimization techniques to compute approximate centerpoints. 
However, this comes with a significant loss in the approximation guarantee.
Recall the definition of $\mu_{mixed,P}$ from~\eqref{eq:mixed}.

\begin{theorem}\label{thm:approximate-mixed-algo}
Let $n,d \in \N$ be fixed and let $P$ be a rational polytope. 
Then in polynomial time in the input-size of $P$, one can find a point 
$$z\in\left\{ x \in \Z^n\times\R^d : f_{\mu_{mixed,P}}(x) \geq \frac{1}{2^{n^2}(d+1)^{(n+1)}} \right\}.$$
\end{theorem}

\begin{proof}
By Theorem~\ref{thm:exact-uniform-algo}, the statement holds for $n=0$. 
Also, since Theorem~\ref{thm:boundForLargeLatticeWidth} is constructive, there exists a $\bar\w$ that only depends on $n$ and $d$, such that the theorem holds true provided that the lattice-width of $P$ is larger than $\bar\w$.

By induction we assume that the result is true for $n-1$.
Further, we may assume that the lattice width is smaller than $\bar\w$.
Without loss of generality, we assume that the flatness direction of $P$ is equal to $n$-th unit vector, i.e., $\min_{x\in P}x_n\ge0$ and $\max_{x\in P}x_n\le\bar\w$. 
We define $P_i:=P\cap\{x\in\R^{n+d}: x_n=i\}$ and the corresponding uniform measure $\mu_i$.
By the induction hypothesis, we can compute $z_i\in \left\{ x \in \Z^n\times\R^d : f_{\mu_{i}}(x) \geq \frac{1}{2^{(n-1)^2}(d+1)^{(n)}} \right\}$
for $i=0,\dots,\bar\w$.

We define the finite auxiliary measure:
$$\bar\mu(x):=
\begin{cases}
\mu(P_i) & \text{ if } x=z_i,\\
0 & \text {otherwise.}
\end{cases}$$
Then, with a brute force approach, we  compute the centerpoint $z$ in $\mathcal{C}(\Z^n\times\R^d,\bar\mu)$. 

It remains to show that $z\in\left\{ x \in \Z^n\times\R^d : f_{\mu_{mixed,P}}(x) \geq \frac{1}{2^{n^2}(d+1)^{(n+1)}} \right\}$.  
Let $H^+$ be any half-space containing $z$.
Note that, for all $i$ we have $\mu(P_i \cap H^+) \ge\frac{1}{2^{(n-1)^2}(d+1)^{n}}\bar\mu(\{z_i\}\cap H^+)$.
Hence,
$$\mu(P\cap H^+)=\sum_{i=0}^{\bar\w}\mu(P_i\cap H^+)\ge\frac{1}{2^{(n-1)^2}(d+1)^{n}}\sum_{i=0}^{\bar\w}\bar \mu(\{z_i\}\cap H^+)\ge\frac{1}{2^{n^2}(d+1)^{n+1}},$$
where the last inequality comes from Theorem~\ref{thm:general:thm:helly}.
\end{proof}

\subsubsection{Computing approximate centerpoints with a Monte-Carlo algorithm} In this section, we compute $\epsilon$-centerpoints, but for any family of measures from which one can sample uniformly. However, now the algorithm's runtime depends polynomially on $\frac1\epsilon$, as opposed to $\log(\frac1\epsilon)$ as for the uniform measure on rational polytopes from Section \ref{sec:exact}.

Suppose we have access to two black-box algorithms:
\begin{enumerate}
\item OPT is an algorithm which works for some family $\mathcal{S}$ of closed subsets of $\R^n$. OPT takes as input a closed set $S\in \mathcal{S}$ and (approximately) computes $\argmax_{x\in S} g(x)$ for any quasi-concave function $g$, given an (approximate) evaluation oracle for $g$ and an (approximate) separation oracle for the sets $\{x  :  g(x) \geq \alpha\}_{\alpha\in \R}$. Let $T_1(S)$ be the number of calls that OPT makes to the evaluation and separation oracles, and $T_2(S)$ be the number of elementary arithmetic operations OPT makes during its execution.
\item SAMPLE is an algorithm which works for some family of probability measures $\Gamma$. SAMPLE takes as input a measure $\mu \in \Gamma$ and produces a sample point $x\in \R^n$ from the measure $\mu$. Let $T(\mu)$ be the running time for SAMPLE.
\end{enumerate}

We now show that with access to the above two algorithms, one can compute an $\epsilon$-centerpoint for $(S, \mu) \in \mathcal{S} \times \Gamma$. 

\begin{theorem}\label{thm:e-approx}
Let $\mathcal{S}$ be a family of closed subsets of $\R^n$ equipped with an algorithm OPT as described above, and let $\Gamma$ be a family of measures on $\R^n$ equipped with an algorithm SAMPLE as described above. 

There exists a Monte Carlo algorithm which takes as input $(S,\mu) \in \mathcal{S}\times \Gamma$, real numbers $\epsilon,\delta>0$ and computes an $\epsilon$-approximate centerpoint for $S,\mu$ with probability at least $1-\delta$. The running time of this algorithm is $T_1(S)\cdot N^n + T_2(S) + T(\mu)\cdot N$, where $N = O(\frac{1}{\epsilon^2}((n+1) + \log\frac{1}{\delta}))$.

\end{theorem}

To prove this theorem, we will need a deep result from probability theory that has resulted after a long line of research sparked by the seminal ideas of Vapnik and Chervonenkis~\cite{vapnik1971uniform}, and culminated in a result of Talagrand~\cite{talagrand1994sharper}. 
The following theorem is a rewording of Talagrand's result~\cite{talagrand1994sharper}, specialized for function classes with bounded VC-dimension.

\begin{theorem}\label{thm:sample-complexity}
Let $(X,\mu)$ be a probability space. Let $\F$ be a family of functions mapping $X$ to $\{0,1\}$ and let $\nu$ be the VC-dimension of the family $\F$. There exists a universal constant $C$, such that for any $\epsilon,\delta >0$, if $M$ is a sample of size $C\cdot\frac{1}{\epsilon^2}(\nu + \log\frac{1}{\delta})$ drawn independently from $X$ according to $\mu$, then with probability at least $1-\delta$, for every function $f\in \F$, $\left|\frac{|\{x\in M : f(x)=1\}|}{|M|} - \mu(\{x\in X :  f(x) = 1\})\right| \leq \epsilon$.
\end{theorem}

\begin{proof}[Proof of Theorem~\ref{thm:e-approx}]
We call SAMPLE to create a sample $M$ of size $C\cdot\frac{1}{\epsilon^2}((n+1) + \log\frac{1}{\delta})$ by drawing independently and uniformly at random from $S$ (note that $M$ may contain multiple copies of the same point from $S$). 
Since the VC-dimension of the family of half spaces in $\R^n$ is $n+1$, we know from Theorem~\ref{thm:sample-complexity} that with probability at least $1-\delta$, for every half space $H^+$, $\left|\frac{|H^+\cap M|}{|M|} - \mu(H^+)\right| \leq \epsilon$. Let $\mu'$ be the counting measure on $M$. Then we obtain that $|f_{\mu'}(x) - f_\mu(x)| \leq \epsilon$ for all $x \in \R^n$. Therefore, any $x^* \in \arg\max_{x\in S}f_{\mu'}(x)$ is an $\epsilon$-centerpoint for $S$. This can be achieved by calling OPT to compute $x^* \in \arg\max_{x\in S}f_{\mu'}(x)$. For this, we need to exhibit evaluation and separation oracles for $f_{\mu'}$. But notice that, by Lemma~\ref{quasi-concave}, this can be accomplished by simply implementing the following procedure: given $x\in \R^d$, find the best hyperplane $H$ through $x$ such that $\frac{|H^+\cap M|}{|M|}$ is minimized. This can be done in time $O(|M|^n)$ by simply enumerating all hyperplanes that contain $n-1$ affinely independent points from $M$.
\end{proof}

The following result is a consequence.

%

\begin{theorem}\label{thm:mixed-integer-e-approx}
Let $n$ and $d$ be fixed integers. There exists a Monte Carlo algorithm which takes as input an integer $m \geq 1$, a matrix $A\in \R^{m\times (n+d)}$, a vector $b\in \R^m$, real numbers $\epsilon, \delta > 0$ and returns an $\epsilon$-approximate centerpoint when 
$S=\Z^n\times\R^d$ and $\mu$ is the uniform measure on $\{x \in \Z^n\times \R^d :  Ax \leq b\}$,
with probability $1-\delta$. The running time of the algorithm is a polynomial in $m, {  \log(\max\{|A_{i,j}|,\,|b_k|\})},\frac{1}{\epsilon},\log\frac{1}{\delta}$. 
\end{theorem}

\begin{proof} By using classical results on maximizing quasi-concave functions over the integer points in a polyhedron~\cite{GroetschelLovaszSchrijver-Book88}, OPT can be implemented for the family $\mathcal{S}$ which is the collection of all sets $S$ that can be represented as the set of mixed-integer points in a rational polytope. For $n=0$, SAMPLE can be implemented for the uniform measure on polytopes using well-studied techniques, e.g., see Vempala's survey~\cite{vempala2010recent}. For $n\geq 1$, SAMPLE can be implemented for the uniform measure on mixed-integer points in a polytope by adapting a result of Igor Pak~\cite{pak2000sampling} on $d=0$ to $d\geq 1$ and using results on computing mixed-integer volumes in polynomial time for fixed dimensions~\cite{baldoni2013intermediate}.
\end{proof}



\section*{Acknowledgments} Parts of the results presented in this paper appear in the second author's dissertation~\cite{oertel2014integer} and in the Proceedings of IPCO 2016~\cite{basu2016centerpoints}. 

We also thank two anonymous referees whose pointers to the literature and numerous suggestions helped to improve the content and presentation of this paper.

\small
\bibliographystyle{plain}
\bibliography{../full-bib}

\end{document}